%% file: ExnerLotoreichik_revised.tex
\documentclass[11pt]{amsart}

\usepackage{lmodern}
\usepackage{eucal}

\usepackage{bm,bbm,amsfonts,amsmath,amssymb,mathrsfs,tikz,hyperref,dsfont,csquotes,enumerate} 
\hypersetup{pdfborder={0000}, colorlinks=true, linkcolor=blue,citecolor=citegreen}
\definecolor{citegreen}{rgb}{0.2,0.2,0.6}

\usepackage[babel=true,kerning=true]{microtype}
\definecolor{darkred2}{rgb}{0.8,0.1,0.1}
\definecolor{darkred}{rgb}{0.8,0.1,0.1}
\usepackage{accents}

\input{commands.tex}

\newcommand\efR{\psi_{\Sigma}}

\def\G{\Gamma}

\title[A spectral isoperimetric inequality for cones]{A spectral isoperimetric inequality for cones}

\author{Pavel Exner}
\address{Department of Theoretical Physics, Nuclear Physics Institute, Czech Academy of Sciences, 25068 \v Re\v z near Prague, Czechia, and Doppler Institute for Mathematical Physics and Applied Mathematics, Czech Technical University, B\v rehov\'a 7, 11519 Prague, Czechia}
\email{exner@ujf.cas.cz} 

\author{Vladimir Lotoreichik}
\address{Department of Theoretical Physics,
Nuclear Physics Institute, Czech Academy of Sciences, 250 68, 
\v{R}e\v{z} near Prague, Czechia}
\email{lotoreichik@ujf.cas.cz}


\begin{document}

\begin{abstract}
	In this note we investigate three-dimensional Schr\"odinger operators
	with $\delta$-interactions supported on $C^2$-smooth cones,
	both finite and infinite. Our main results concern a Faber-Krahn-type
	inequality for the principal eigenvalue of these operators.
	The proofs rely	on the Birman-Schwinger principle and on
	the fact that circles are unique minimizers for a class
	of energy functionals. The main novel idea consists
	in the way of constructing test functions for the Birman-Schwinger
	principle.
\end{abstract}



\maketitle

\section{Introduction and results}
Relations between geometric properties of the domains and
spectral properties of partial differential operators acting on them
belong to the trademark topics in mathematical physics. 
Spectral isoperimetric inequalities are one of the most famous examples 
of such relations, the first rigorous results dating almost a 
century back to the papers of Faber~\cite{Fa23} and Krahn~\cite{Kr25}. 
Recently spectral isoperimetric inequalities appeared in the context 
of Schr\"odinger operators with singular potentials
used as models of `leaky quantum wires' and similar systems~\cite{E08}, \cite[Chap.~10]{EK15}. In 
particular, for the two-dimensional Schr\"odinger operator with a 
$\delta$-type potential of a fixed strength supported on a loop of a given 
length it was shown that its principal eigenvalue is maximal when the loop 
is a circle, the respective isoperimetric inequality being strict~\cite{EHL06}. The corresponding problem in three dimensions is 
more involved. For closed simply connected surfaces of a fixed area the sphere gives a local maximum of the ground-state eigenvalue, however, the result does not have a global validity~\cite{EF09}.

Nevertheless, there are three-dimensional Schr\"odinger operators with
singular interactions supported on surfaces for which one is able to derive a
spectral isoperimetric inequality that holds not only locally. The aim of the 
present paper is to analyse one such class. The surfaces in question are
of a conical shape, both finite and infinite. The operators of study
are fully described through the strength $\aa > 0$ of the $\delta$-interaction, 
the radius of the cone $R \in (0,+\infty]$,
and its cross-section, whose length $L\in(0,2\pi]$
is important for our considerations.

For finite cones ($R < \infty$) we first verify that $\delta$-interactions
supported on finite circular cones (\ie having a rotational
symmetry) induce at least one negative bound
state if, and only if, the strength of the interaction
satisfies $\aa > \aa_{\rm cr}$ with certain
$\aa_{\rm cr} = \aa_{\rm cr}(L, R) > 0$.
Furthermore, we show that $\delta$-interactions supported
on finite non-circular cones with the same length of the cross-section 
and the same radius induce at least one negative bound state for any strength 
$\aa \ge \aa_{\rm cr}$. It should be stressed and it is non-trivial to 
show that for non-circular cones at least one negative bound state exists also in the borderline case
$\aa = \aa_{\rm cr}$. As the main result for finite cones
we prove that for any fixed set of parameters:
$L\in (0,2\pi]$, $R> 0$, and $\aa > \aa_{\rm cr}$,
the principal eigenvalue is maximized by circular cones
supporting the interaction. Moreover, the respective
spectral isoperimetric inequality is strict.

For infinite cones ($R = +\infty$) 
we verify that the discrete spectrum below the threshold
of the essential spectrum is always non-empty for any $L \in (0,2\pi)$
and $\aa > 0$. As the main result we prove that for fixed
$L\in(0,2\pi)$ and $\aa > 0$ 
the principal eigenvalue is maximized by infinite circular cones.

Spectral analysis of Schr\"odinger operators with $\delta$-interactions on 
conical surfaces and of closely related Robin Laplacians on conical domains
has attracted considerable attention in the recent time 
\cite{BEL14_JPA, BP15, ER15, LP08, LO15, LR15, P15}.
Finally, we mention that spectral isoperimetric inequalities were 
previously known
for other classes of partial differential operators with 
singular interactions; see
\cite{Ex05,Ex06} for Schr\"odinger operators with point interactions,
\cite{Da06, FK15}  for Robin Laplacians,  \cite{AMV15}  for Dirac operators
with shell-interactions, and \cite{BFKLR} for Schr\"odinger operators
with $\delta$-interactions supported on curves in $\dR^3$.

\subsection{Definition of the Hamiltonian}

To define the operators of study we first introduce notations
for some standard function spaces.
The $L^2$-spaces $(L^2(\dR^3), (\cdot,\cdot)_{\dR^3})$, 
$(L^2(\dR^3;\dC^3), (\cdot,\cdot)_{\dR^3})$
and the $L^2$-based Sobolev space $H^1(\dR^3)$ are defined in the usual way.
For a Lipschitz surface $\Sigma\subset\dR^3$,
which is not necessarily closed or bounded (\cf~\cite[Sec. 2.3]{BEL14_RMP})
we define the $L^2$-space 
$(L^2(\Sigma), (\cdot,\cdot)_\Sigma)$ by means of the 
natural surface measure on $\Sigma$. 
For $u\in H^1(\dR^3)$ its trace $u|_{\Sigma}$ is well-defined
as a function in $L^2(\Sigma)$.

Let $\aa  >  0$ be a fixed constant and let $\Sigma\subset\dR^3$
be a Lipschitz surface.
According to~\cite[Sec. 2]{BEKS94} (see also \cite[Prop. 3.1]{BEL14_RMP}) the symmetric densely defined 
quadratic form
\begin{equation}\label{def:frm}
	\fra_{\aa,\Sigma}[u] := 
	\|\nabla u\|_{\dR^3}^2 - \alpha \|u|_\Sigma\|^2_{\Sigma},
	\qquad
	\dom \fra_{\aa,\Sigma} := H^1(\dR^3),
\end{equation}	
is closed and lower-semibounded in $L^2(\dR^3)$.
%
%
\begin{dfn}\label{dfn:Op}
	The self-adjoint operator $\sfH_{\aa,\Sigma}$ acting in $L^2(\dR^3)$ 
	associated to the form $\fra_{\aa,\Sigma}$ in~\eqref{def:frm}
	via the first representation theorem (\cite[Thm. VI.2.1]{Kato})
	is called \emph{Schr\"odinger operator with $\delta$-interaction}
	of strength $\aa > 0$ supported on $\Sigma$.
\end{dfn}

\subsection{Cones}
In our considerations cones serve as supports
of $\delta$-interactions. Further, we explain what we understand by cones
in the present paper.

Let $\cT \subset \dS^2$ be a $C^2$-smooth loop
on the two-dimensional 
unit sphere $\dS^2\subset\dR^3$.
The length of $\cT$ is denoted by $|\cT|$.
It is always implicitly assumed that $\cT$ has no self-intersections.
We distinguish between \emph{circular loops} (or simply \emph{circles})
and \emph{non-circular loops}. 
A circle on $\dS^2$ will be occasionally denoted by $\cC$, 
and we  point out that $|\cC|\le 2\pi$.

The $C^2$-smooth \emph{conical surface} (or simply \emph{cone}) 
$\Sigma_R(\cT) \subset \dR^3$ of radius $R \in ( 0, +\infty ]$ 
with a $C^2$-smooth loop $\cT\subset\dS^2$ as the~\emph{cross-section}  is defined by 
\begin{equation}\label{def:S}
	 \Sigma_R(\cT) := \big\{ r \cT \in\dR^3 \colon r\in [0,R) \big\}.
\end{equation}	  
The cone $\Sigma_R(\cT)$ is called \emph{finite} (or \emph{truncated}) if $R < \infty$ and
\emph{infinite} if $R = +\infty$, respectively. 
The cross-section of $\Sigma_R(\cT)$ can be easily recovered by the formula 
$\cT = \dS^2\cap\Sigma_R(\cT)$. 
The cone $\Sigma_R(\cT)$ is called \emph{circular} if its 
cross-section $\cT$ is a circle and \emph{non-circular}, otherwise. 
We remark also that infinite circular cone with the 
cross-section of length $2\pi$ is, in fact, a plane.
Finally, note that $\Sigma_R(\cT)$ is, in particular, a Lipschitz surface.

\subsection{Main results}

In the following, for a lower-semibounded
self-adjoint operator $\sfH$ we denote by $E_1(\sfH)$ its lowest
eigenvalue (if it exists);
the discrete and essential spectra of $\sfH$ will be denoted by
$\sd(\sfH)$ and $\sess(\sfH)$, respectively. We denote by $\sfP_\cI(\sfH)$ the spectral projector for $\sfH$ corresponding to a Borel set  $\cI\subset\dR$.
By $\# \cM$ we understand
the cardinality of a discrete set $\cM$. 

It follows from the results in \cite{BEL14_RMP, BEKS94}
that the essential spectrum of Schr\"odinger operators with $\delta$-interactions
supported on finite $C^2$-smooth cones coincides with the set $[0,+\infty)$. 
In the first theorem of the paper we collect our main results on
the bound states induced by $\delta$-interactions supported on finite cones.
%
\begin{thm}\label{thm1}
	Let $\cC\subset\dS^2$ be a circle and  
	$\cT \subset\dS^2$ be a $C^2$-smooth non-circular loop
	such that $L := |\cC| = |\cT| \in (0,2\pi]$. 
	Let $\Gamma_R := \Sigma_R(\cC)$ and $\Lambda_R := \Sigma_R(\cT)$
	be finite cones of radius $R > 0$ as in~\eqref{def:S}
	with the cross-sections $\cC$ and $\cT$,
	respectively.
	Let the self-adjoint operators $\sfH_{\aa,\Gamma_R}$ and $\sfH_{\aa,\Lambda_R}$ 
	be as in Definition~\ref{dfn:Op}. Then the following hold.
	\begin{myenum}
		\item $\#\sd(\sfH_{\aa,\Gamma_R}) \ge 1$ 
		if, and only if, $\aa > \aa_{\rm cr}$ for certain
		$\aa_{\rm cr} = \aa_{\rm cr}(L,R) > 0$.
		\item  $\#\sd(\sfH_{\aa,\Lambda_R}) \ge 1$ for all $\aa \ge \aa_{\rm cr}$
		(the borderline case $\aa = \aa_{\rm cr}$ is included)
		and the spectral isoperimetric inequality 
		\[		
			E_1(\sfH_{\aa,\Lambda_R}) < E_1(\sfH_{\aa,\Gamma_R})
		\] 
		is satisfied for all $\aa > \aa_{\rm cr}$.
	\end{myenum}
\end{thm}
The strategy of the proof of Theorem~\ref{thm1} consists in reducing
the spectral problems for Schr\"odinger operators $\sfH_{\aa,\Gamma_R}$ and $\sfH_{\aa,\Lambda_R}$ to spectral problems for operator-valued functions acting 
in $L^2$-spaces over respective cones. For this reduction 
we employ a generalization of the Birman-Schwinger 
principle~\cite{B95, BEKS94}. 
In further analysis, a crucial role is played by the result that circles are
unique minimizers for certain classes of knot energies~\cite{ACGJ03, EHL06}.

It follows from the results of~\cite{BEL14_JPA, BP15, P15} that 
$-\aa^2/4$ is the lowest point of the essential spectrum
for a Schr\"odinger operator with $\delta$-interaction
of strength $\aa > 0$ supported on an infinite $C^2$-smooth cone. 
It is also proven in~\cite{BEL14_JPA} 
that the discrete spectrum below the point $-\aa^2/4$
in the case of infinite circular cones with the cross-section of length $L\in (0, 2\pi)$
is non-empty (and even infinite). Further analysis of this discrete spectrum
is carried out in~\cite{LO15}. In the second theorem of the paper we 
collect our main results on the bound states induced by $\delta$-interactions
supported on infinite cones. 
\begin{thm}\label{thm2}
	Let $\cC\subset\dS^2$ be a circle and  
	$\cT \subset\dS^2$ be a non-circular $C^2$-smooth loop
	such that $|\cC| = |\cT| \in (0,2\pi)$.
	Let $\Gamma_\infty := \Sigma_\infty(\cC)$ and 
	$\Lambda_\infty := \Sigma_\infty(\cT)$
	be infinite cones as in~\eqref{def:S}
	with the cross-sections $\cC$ and $\cT$,
	respectively.
	Let the self-adjoint operators $\sfH_{\aa,\Gamma_\infty}$ and 
	$\sfH_{\aa,\Lambda_\infty}$ be as in Definition~\ref{dfn:Op}. 
	Then for all $\aa > 0$ the following hold.
	\begin{myenum}
		\item $\#\sd( \sfH_{\aa,\Lambda_\infty}) \ge 1$.
		\item The spectral isoperimetric inequality 
		$E_1( \sfH_{\aa,\Lambda_\infty}) \le E_1( \sfH_{\aa,\Gamma_\infty})$
		is satisfied.				
	\end{myenum}	
\end{thm}
The key idea of the proof is to consider $\delta$-interactions
supported on $C^2$-smooth finite cones $\Gamma_R := \Sigma_R(\cC)$ 
and $\Lambda_R := \Sigma_R(\cT)$ 
of radius $R > 0$.
Using convergence results for monotone families of quadratic forms, 
we show that $\sfH_{\aa,\Gamma_R}$ and $\sfH_{\aa,\Lambda_R}$ 
converge (as $R \arr \infty$)
in the strong resolvent sense to $\sfH_{\aa,\Gamma_\infty}$ 
and $\sfH_{\aa,\Lambda_\infty}$, respectively. 
Finally, we combine \cite[Thm. 3.2]{BEL14_JPA}
with some standard results on spectral convergence
and with the spectral isoperimetric inequality 
in Theorem~\ref{thm1}\,(ii) to get both the
statements of Theorem~\ref{thm2}.
\begin{remark}\label{rem:strong}	
	Let the notations be the same as in formulations
	of Theorems~\ref{thm1} and~\ref{thm2}.
	As a byproduct of the strong resolvent convergence described above
	and of standard spectral convergence
	results~\cite[Satz 2.58\,(a) and Satz 9.24\,(b)]{W13}, 
	one can show that for any closed interval $\cI\subset (-\infty,-\aa^2/4)$ the functions 
	\begin{equation}\label{eq:functions}
		R\mapsto \dim\ran \sfP_\cI(\sfH_{\aa,\Gamma_R})
		\qquad\text{and}
		\qquad R\mapsto \dim\ran \sfP_\cI(\sfH_{\aa,\Lambda_R})
	\end{equation}	
	are bounded.  
	Moreover, according to~\cite[Thm. 2.1]{BEL14_JPA}
	we also have $\sess(\sfH_{\aa,\Gamma_\infty}) = [-\aa^2/4,\infty)$ and,
	thus, for any closed interval $\cI\subset (-\aa^2/4,0)$ 
	the first function in~\eqref{eq:functions} tends
	to infinity as $R\rightarrow\infty$. We expect that 
	the second function in~\eqref{eq:functions} does the same.
	However, in order to show this, it is insufficient to know only that $\inf\sess(\sfH_{\aa,\Lambda_\infty})= -\aa^2/4$, one needs to verify 
	$\sess(\sfH_{\aa,\Lambda_\infty}) = [-\aa^2/4,\infty)$. 
	To the best of our knowledge, the latter is not contained in the existing literature, but it can be shown in a way similar to~\cite[Thm. 1]{P15}.

	Summing up,
	as the radius grows,
	the number of negative
	eigenvalues	counted with multiplicities for $\delta$-interactions
	on finite cones	remains bounded in any closed sub-interval of $(-\infty,-\aa^2/4)$
	and tends to infinity in any closed sub-interval of $(-\aa^2/4,0)$. 
\end{remark}
We  conclude this subsection by some open questions. Recall that according to Theorem~\ref{thm2}\,(i) the discrete spectrum of $\sfH_{\aa,\Lambda_\infty}$
is non-empty if $|\Lambda_\infty\cap\dS^2| < 2\pi$. By analogy 
with~\cite{DEK01} one may conjecture that the discrete spectrum is also non-empty if $|\Lambda_\infty\cap\dS^2| \ge 2\pi$,
of course, excluding the case when $\Lambda_\infty$ is a plane. 
In view of~\cite[Thm. 3.2]{BEL14_JPA}, 
it presents a subtle problem to check whether the discrete
spectrum of $\sfH_{\aa,\Lambda_\infty}$ is finite or infinite.
We expect that the discrete spectrum is infinite for  $|\Lambda_\infty\cap\dS^2| < 2\pi$. On the other hand, for $|\Lambda_\infty\cap\dS^2| \ge 2\pi$
it might be that the discrete spectrum is finite, at least under some extra
assumptions on the shape of $\Lambda_\infty\cap\dS^2$. 
\subsection*{Organization of the paper}

Section~\ref{sec:prelim} contains some preliminary material that will be
used in the proofs of Theorems~\ref{thm1} and~\ref{thm2}.
Namely, in Subsection~\ref{ssec:BS} we provide Birman-Schwinger principle and
prove some related statements. In Subsection~\ref{ssec:param} we introduce natural coordinates on cones
and derive some consequences of this parametrization. Energies of knots and
their minimizers are briefly discussed in Subsection~\ref{ssec:knots}.
Section~\ref{sec:proofs} contains proofs of the main results together with some auxiliary statements: Theorem~\ref{thm1} is proven in Subsection~\ref{ssec:proof1} and Theorem~\ref{thm2} in Subsection~\ref{ssec:proof2}.

\section{Preliminaries}
\label{sec:prelim}

\subsection{Birman-Schwinger principle}
\label{ssec:BS}
The BS-principle is a classical and a powerful tool for the spectral analysis of 
Schr\"odinger operators. Its generalization, which encompasses 
$\delta$-interactions supported on hypersurfaces, is derived in~\cite{BEKS94}; see also~\cite{BLL13_AHP} and~\cite{B95}.

Let $\lm \le 0$ and set $\kp := \sqrt{-\lm}$. Green's function corresponding to the 
differential expression $-\Delta + \kp^2$ in $\dR^3$ has the following well-known form
\begin{equation} \label{freeG3}
	G_\kp(x\!-\!y) = \frac{e^{-\kp|x-y|}}{4\pi | x \!- \! y|}.
\end{equation}
Let $\Sigma\subset\dR^3$ be a compact Lipschitz surface,
which is not necessarily closed; \cf \cite[Sec. 2.3]{BEL14_RMP}.
Further, we introduce the following mapping
\begin{equation}\label{def:SL}
	\big(\sfS_{\Sigma}(\kp) \psi\big)(x) 
	:= 
	\int_{\Sigma} G_\kp(x-y)\psi(y)\dd \s(y),
\end{equation}
where $\dd \s$ is the natural surface measure on $\Sigma$.
The mapping $\sfS_\Sigma(\kp)$, $\kp \ge 0$, extends to a compact
operator in $L^2(\Sigma)$; \cf \cite[Rem. 2.1, Lem. 3.2]{BEKS94} 
and~\cite[Sec. 2]{Go12}.
As there is no danger of confusion, we denote this extension again by~$\sfS_{\Sigma}(\kp)$.
According to~\cite{B95} the operator $\sfS_{\Sigma}(\kp)$, $\kp > 0$,
is self-adjoint and non-negative.
We make use of the following hypothesis to shorten the formulations.
\begin{hyp}\label{hypothesis}
	Let $\aa > 0$ and let
	$\Sigma\subset\dR^3$ be a compact Lipschitz surface as above.
	Let the self-adjoint operator $\sfH_{\aa,\Sigma}$ be as in 
	Definition~\ref{dfn:Op} and the operator-valued function 
	$[0,+\infty) \ni \kp \mapsto \sfS_{\Sigma}(\kp)$
	be as in~\eqref{def:SL}.
	Let $\mu_\Sigma(\kp) > 0$ be the largest eigenvalue of $\sfS_\Sigma(\kp)$
	and 
	$E_1(\sfH_{\aa,\Sigma}) < 0$ be the lowest eigenvalue of $\sfH_{\aa,\Sigma}$
	(if it exists).
\end{hyp}
The essential spectrum of $\sfH_{\aa,\Sigma}$ coincides with the
set $[0,+\infty)$; \cf \cite[Thm. 3.2]{BEKS94}.
The next theorem contains the BS-principle 
for the negative spectrum of the operator $\sfH_{\aa,\Sigma}$; for the proof 
see~\cite[Lem. 2.3\,(iv)]{BEKS94}. 
\begin{thm}\label{thm:BS}	
	Assume that Hypothesis~\ref{hypothesis} holds.
	Then the relation
	\[
		\dim\ker\big(\sfH_{\aa,\Sigma} + \kp^2\big) 
		= 
		\dim\ker\big( I - \aa \sfS_\Sigma(\kp) \big)
	\]	
	is satisfied for all $\kp > 0$.
\end{thm}
%
In the next simple lemma 
we characterize the bottom of the spectrum
of $\sfH_{\aa,\Sigma}$ as a function of $\aa$.
\begin{lem}\label{lem:cont}
	Assume that Hypothesis~\ref{hypothesis} holds.
	Then the function $[0,+\infty) \ni \aa \mapsto F_\Sigma(\aa) := 
	\inf\s(\sfH_{\aa,\Sigma})$ has the following properties.
	\begin{myenum}
		\item $F_\Sigma$ is continuous and non-increasing.
		\item $\ran F_\Sigma = (-\infty,0]$.
		\item If $F_\Sigma(\aa) < 0$, then $F_\Sigma(\aa') < F_\Sigma(\aa)$
		for any $\aa' > \aa$.
	\end{myenum}	
\end{lem}
\begin{proof}
	The statement of~(i) is a consequence of~\cite[Lem. 3.3]{BEKS94}.

	To show~(ii) it suffices to note that $F_\Sigma(0) = 0$
	and that $\lim_{\aa \arr +\infty} F_\Sigma(\aa) = -\infty$.
	The former is obvious and to show the latter we
	observe that by the min-max principle 
	(see \eg \cite[\S 10.2, Thm. 4]{BS87} or \cite[Thm. 14.2.1]{BEH})
	\begin{equation}\label{eq:ineq}
		F_\Sigma(\aa) \le \fra_{\aa,\Sigma}[\chi]
	\end{equation}	
	for any $\chi \in C^\infty_0(\dR^3)$ such that
	$\|\chi\|_{\dR^3} = 1$. Choose now such $\chi$
	so that	$\|\chi|_{\Sigma}\|_{\Sigma} > 0$
	and pass to the limit $\aa\arr +\infty$ in~\eqref{eq:ineq}.

	To prove~(iii) we pick the normalized
	ground-state eigenfunction $\psi_1$ of 
	$\sfH_{\aa,\Sigma}$ 
	and plug it into the quadratic form $\fra_{\aa',\Sigma}$.
	The inequality $\fra_{\aa',\Sigma}[\psi_1] < \fra_{\aa,\Sigma}[\psi_1]$
	holds and the min-max principle yields the claim.
\end{proof}
In the next proposition we derive a consequence of 
the BS-principle and of the min-max characterization for the principal
eigenvalue of $\sfH_{\aa,\Sigma}$.
\begin{prop}\label{lem:BS1}
	Assume that Hypothesis~\ref{hypothesis} holds.
	Then for $\kp > 0$ the following statements hold.
	\begin{myenum}
		\item $\#(\sd(\sfH_{\aa,\Sigma})\cap (-\infty,-\kp^2) ) \ge 1$ 
		if, and only if, $\mu_\Sigma(\kp) > 1/\aa$.
		\item $E_1(\sfH_{\aa,\Sigma}) = -\kp^2$ if, and only if, 
		$\mu_\Sigma(\kp) = 1/\aa$. 
	\end{myenum}	
\end{prop}

\begin{proof}
	We split the proofs of both items into showing two implications.

	{\rm (i)}
	To show ``$\Rightarrow $'' we suppose that
	$\#(\sd(\sfH_{\aa,\Sigma})\cap (-\infty,-\kp^2) ) \ge 1$ 
	and thus  $E_1(\sfH_{\aa,\Sigma}) < -\kp^2$.
	Hence, by Lemma~\ref{lem:cont} there exists
	$\aa' \in (0, \aa)$ such that $E_1(\sfH_{\aa',\Sigma}) = -\kp^2$.
	In view of Theorem~\ref{thm:BS}, we have, in particular, 
	$1/\aa'\in \sd(\sfS_{\Sigma}(\kp))$. Thus, we arrive at
	$\mu_\Sigma(\kp) \ge 1/\aa' > 1/\aa$.

	To verify ``$\Leftarrow $'' we  presume that $\mu_\Sigma(\kp) > 1/\aa$
	and set $\aa' := (\mu_\Sigma(\kp))^{-1}$. 
	In particular, we have $\aa' \in (0, \aa)$ and also $-\kp^2\in\sd(\sfH_{\aa',\Sigma})$ by Theorem~\ref{thm:BS}. 
	Lemma~\ref{lem:cont} implies that
	$E_1(\sfH_{\aa,\Sigma}) < E_1(\sfH_{\aa',\Sigma}) \le -\kp^2$
	and that $\#(\sd(\sfH_{\aa,\Sigma})\cap (-\infty,-\kp^2) ) \ge 1$. 

	{\rm (ii)}
	For ``$\Rightarrow $'', we suppose that $E_1(\sfH_{\aa,\Sigma}) = -\kp^2$. 
	Then by Theorem~\ref{thm:BS} we have
	$1/\aa\in\sd(\sfS_\Sigma(\kp))$. Therefore, we get
	$\mu_\Sigma(\kp) \ge 1/\aa$, but the inequality
	$\mu_\Sigma(\kp) > 1/\aa$ leads to a contradiction to item~(i)
	of this proposition, 
	because we would obtain $E_1(\sfH_{\aa,\Sigma}) < -\kp^2$. 
	Therefore, the equality $\mu_\Sigma(\kp) = 1/\aa$ is satisfied.	
	
	For ``$\Leftarrow $'', we presume that $\mu_\Sigma(\kp) = 1/\aa$. 
	Then again by Theorem~\ref{thm:BS} 
	we have	$-\kp^2\in\sd(\sfH_{\aa,\Sigma})$
	and hence $E_1(\sfH_{\aa,\Sigma}) \le -\kp^2$.
	The inequality $E_1(\sfH_{\aa,\Sigma}) < -\kp^2$ leads to a contradiction
	to item~(i), because we would get $\mu_\Sigma(\kp) > 1/\aa$.
	Therefore, the equality $E_1(\sfH_{\aa,\Sigma}) = -\kp^2$ holds.
\end{proof}	
Next, we analyse the dependence of $\mu_\Sigma(\kp)$
on the parameter $\kp$.
\begin{lem}\label{lem:cont2}
	Assume that Hypothesis~\ref{hypothesis} holds.
	Then the function $(0,+\infty) \ni \kp \mapsto \mu_\Sigma(\kp)$
	is continuous and decreasing.
\end{lem}

\begin{proof}
	{\rm (i)} 
	Continuity of $\mu_\Sigma(\cdot)$ follows from \cite[Lem. 3.2]{BEKS94}
	and its proof.
		
	{\rm (ii)} Let $\kp_1 > \kp_2$ and set $\aa_j := (\mu_\Sigma(\kp_j))^{-1}$, 
	$j=1,2$. Using Lemma~\ref{lem:cont} and Proposition~\ref{lem:BS1} we get
	\[
		F_\Sigma(\aa_1) = E_1(\sfH_{\aa_1,\Sigma}) 
		= 
		-\kp^2_1 < -\kp^2_2 
		= 
		E_1(\sfH_{\aa_2,\Sigma}) 
		= 
		F_\Sigma(\aa_2),
	\]	
	and, hence, $\aa_1 > \aa_2$. Thus,
	$\mu_\Sigma(\kp_1) < \mu_\Sigma(\kp_2)$ and the claim is shown.
\end{proof}
Further, we analyse the behavior of $\sfS_\Sigma(\kp)$ and of $\mu_\Sigma(\kp)$ in the limit
$\kp\arr 0+$.
\begin{lem}\label{lem:conv0}
	Assume that Hypothesis~\ref{hypothesis} is satisfied.
	Then the convergence
	\begin{equation}\label{eq:S_conv}
		\lim_{\kp\arr 0+}\|\sfS_\Sigma(\kp) -\sfS_\Sigma(0)\| = 0
	\end{equation}
        holds, in particular, the operator $\sfS_\Sigma(0)$ is 
	self-adjoint and non-negative, and,
	moreover, $\mu_\Sigma(\kp)\arr \mu_\Sigma(0)$ as $\kp \arr 0+$.
\end{lem}
\begin{proof}
	Introduce the following parameters characterizing $\Sigma$
	\[
		D(\Sigma) := \sup_{x,y\in\Sigma} |x - y| <\infty,
		\qquad
		C(\Sigma) := \sup_{x\in \Sigma} 
		\int_\Sigma \frac{\dd\s(y)}{4\pi|x-y|}< \infty.
	\]
	While finiteness of $D(\Sigma)$ follows directly from compactness of $\Sigma$,
	finiteness of $C(\Sigma)$ is more subtle 
	and is connected with the regularity of $\Sigma$;
	see \cite[Prop. 2]{Go12} for a proof.
	Using  Schur's test \cite[Lem. 0.32]{Te} and the symmetry
	of the integral kernel of $\sfS_\Sigma(\kp) -\sfS_\Sigma(0)$ we obtain
	\[
	\begin{split}
		\|\sfS_\Sigma(\kp) -\sfS_\Sigma(0)\| 
		&\le 
		\sup_{x\in\Sigma}
		\int_\Sigma \frac{1 - e^{-\kp|x-y|} }{4\pi |x-y|} \dd\s(y)\\
		& \le \Big(1-e^{-\kappa D(\Sigma)}\Big)
		\sup_{x\in\Sigma}
		\int_\Sigma \frac{\dd\s(y)}{4\pi|x-y|}
		\le C(\Sigma) \Big(1-e^{-\kp D(\Sigma)}\Big)
		\arr 0,\quad \kp\arr 0+.
	\end{split}	
	\]
	Self-adjointness and non-negativity of $\sfS_\Sigma(0)$
	are thus consequences of respective properties of $\sfS_\Sigma(\kp)$
	for $\kp > 0$ and of the above convergence. Finally,
	$\mu_\Sigma(\kp)\arr \mu_\Sigma(0)$ is equivalent to
	$\|\sfS_\Sigma(\kp)\|\arr \|\sfS_\Sigma(0)\|$ as $\kp \arr 0+$,
	which also follows from~\eqref{eq:S_conv}.
\end{proof}
The statement of the next proposition extends~\cite[Prop. 6.1]{EF09}
to surfaces of lower smoothness.
%
\begin{prop}\label{prop:existence}
	Assume that Hypothesis~\ref{hypothesis} holds.
	Then $\#\sd(\sfH_{\aa,\Sigma})\ge 1$ if, and only if, $\mu_\Sigma(0) > 1/\aa$.
\end{prop}
\begin{proof}
To show ``$\Rightarrow$'' we suppose that $\#\sd(\sfH_{\aa,\Sigma})\ge 1$
and let $E_1(\sfH_{\aa,\Sigma}) = -\kp^2 < 0$ be the corresponding lowest eigenvalue
of $\sfH_{\aa,\Sigma}$. By Lemma~\ref{lem:BS1}\,(ii) 
we get $\mu_\Sigma(\kp) = 1/\aa$. Using Lemmas~\ref{lem:cont2} and~\ref{lem:conv0}
we obtain $\mu_\Sigma(0) > \mu_\Sigma(\kp) = 1/\aa$.

To prove ``$\Leftarrow$'' we suppose that $\mu_\Sigma(0) > 1/\aa$.
Then, according to Lemmas~\ref{lem:cont2} and~\ref{lem:conv0},
for all sufficiently small $\kp > 0$ we have $\mu_\Sigma(\kp) > 1/\aa$.
Hence, by Proposition~\ref{lem:BS1}\,(i) we obtain 
$\#\sd(\sfH_{\aa,\Sigma})\ge 1$. 
\end{proof}
A useful addendum to the BS-principle is provided
in the proposition given below. Its proof is based on a rather standard argument, 
which can be found in some textbooks (see \eg~\cite[Thm. 6.40]{H13}). 
We provide this proof for the sake of completeness.
\begin{prop}\label{prop:BS}
	Assume that Hypothesis~\ref{hypothesis} holds.
	Then the largest eigenvalue of $\sfS_{\Sigma}(\kp)$, $\kp \ge 0$,
	is simple and the corresponding eigenfunction $\efR$ can be chosen
	to be positive almost everywhere on $\Sigma$.
\end{prop}

\begin{proof}
	Since $\sfS_\Sigma(\kp)$ maps real functions
	into real functions, we may assume that $\efR$
	is real-valued. We now show that
	\begin{equation}\label{eq:ineq_Helffer}
		( \sfS_\Sigma(\kp) \efR, \efR )_\Sigma
		\le 
		( \sfS_\Sigma(\kp)| \efR |, | \efR |)_\Sigma.
	\end{equation}
	Let us write
	\[
		\efR = \efR^+ - \efR^-
		\qquad\text{and}\qquad
		|\efR| = \efR^+ + \efR^-,
	\]
	where $\efR^+$ and $\efR^-$ are positive and negative parts
	of $\efR$, respectively.
	The inequality~\eqref{eq:ineq_Helffer} is then a consequence of
	\[
		(\sfS_\Sigma(\kp)\efR^+,\efR^-)_\Sigma \ge 0,
	\]
	which is true, thanks to the positivity of the integral 
	kernel in~\eqref{def:SL}. We then obtain
	\[
	\begin{split}
		\mu_\Sigma(\kp) \|\efR\|^2_{\Sigma} 
		& = 
		(\sfS_{\Sigma}(\kp)\efR,\efR)_{\Sigma}
		\le 
		(\sfS_{\Sigma}(\kp)|\efR|,|\efR|)_{\Sigma}\\
	        &\le
		\|\sfS_{\Sigma}(\kp)\| \|\efR\|^2_{\Sigma}
		= 
		\mu_\Sigma(\kp)\|\efR\|^2_{\Sigma}.
	\end{split}
	\]
	This implies
	\[	
		(\sfS_{\Sigma}(\kp) \efR,\efR)_{\Sigma}
		=
		(\sfS_{\Sigma}(\kp) |\efR|,|\efR|)_{\Sigma}.
	\]
	The above equality yields
	\[
		(\sfS_{\Sigma}(\kp)\efR^+,\efR^-)_{\Sigma} 
		+ 
		(\efR^+,\sfS_{\Sigma}(\kp)\efR^-)_{\Sigma} = 0.
	\]
	Since the integral kernel of $\sfS_\Sigma(\kp)$ is 
	pointwise positive on $\Sigma\times\Sigma$ (\cf \eqref{def:SL}),
	we obtain a contradiction unless either	$\efR^+ = 0$ or $\efR^- = 0$.
	We can assume $\efR \ge 0$ for definiteness. 
	Note that 
	\[
		\efR = (\mu_\Sigma(\kp))^{-1}\sfS_\Sigma(\kp)\efR.
	\]
	This yields $\efR > 0$ almost everywhere on $\Sigma$, again
	because of positivity of the integral kernel of $\sfS_\Sigma(\kp)$.		

	Finally, if the largest eigenvalue of 
	$\sfS_{\Sigma}(\kp)$ were not simple,
	then one would be able to find two orthogonal  
	eigenfunctions $\efR$ and $\varphi_\Sigma$ of $\sfS_\Sigma(\kp)$ 
	corresponding to $\mu_\Sigma(\kp)$. Analogously to the above argument, 
	we would obtain that $\varphi_\Sigma$ is also positive almost everywhere on $\Sigma$
	(up to multiplication with $-1$). 
	But it is impossible to have two orthogonal functions in $L^2(\Sigma)$
	that are both positive almost everywhere.
\end{proof}

\subsection{A parametrization of cones and its consequences}
\label{ssec:param}

Let the cone $\Sigma_R = \Sigma_R(\cT)\subset\dR^3$ 
be as in~\eqref{def:S} with $L := |\cT|$. In this subsection
we provide an efficient parametrization of $\Sigma_R$
and derive some consequences of it. First, note that the cross-section $\cT$
of $\Sigma_R$ can be parametrized by its arc-length via the
unit-speed $C^2$-mapping $\tau \colon [0,L] \arr \dS^2$ ($|\dot\tau| \equiv 1$). 
The cone $\Sigma_R$ can be correspondingly parametrized via the mapping
\begin{equation}\label{eq:S_param}
	\s \colon [0,R) \times [0,L] \arr \dR^3, \qquad \s(r,s) := r\tau(s).
\end{equation}
This parametrization defines natural co-ordinates $(r,s)$ on $\Sigma_R$.

Let the space $L^1(\Sigma_R)$ be introduced as usual, 
by means of the surface measure on $\Sigma_R$.
A function $\psi \in L^1(\Sigma_R)$ can be viewed as a function
of the arguments $r$ and $s$ via the parametrization~\eqref{eq:S_param}.
Thus, the Lebesgue surface integral of $\psi$ can be written as 
\begin{equation}\label{eq:S_integral}
\begin{split}
	\int_{\Sigma_R}  \psi(x) \dd \s(x) & 
	= 
	\int_0^R\int_0^L \psi(r,s) |\s_r(r,s) \times \s_s(r,s)| \dd s \dd r\\
	& = 
	\int_0^R\int_0^L\psi(r,s) |\s_r(r,s)|\cdot |\s_s(r,s)|\dd s \dd r 
	=
	\int_0^R\int_0^L\psi(r,s) r \dd s \dd r;
\end{split}	
\end{equation}	
where, firstly, we employed that the vector $\s_s(r,s) = r\dot\tau(s)$ is of length $r > 0$ 
and belongs to the tangent plane $T_{\tau(s)}(\dS^2)$ of $\dS^2$ at the point $\tau(s)$, secondly, we used that the vector $\s_r(r,s) = \tau(s)$ 
is of unit length and for simple geometric reasons is orthogonal to 
$T_{\tau(s)}(\dS^2)$. A direct consequence of~\eqref{eq:S_integral}
is the following isomorphism 
$L^2(\Sigma_R) \simeq L^2((0,R);r\dd r)\otimes L^2(\cT)$.

Further, with the aid of
the identity 
$|\tau(s) - \tau(s')|^2 = 2 - 2\langle \tau(s), \tau(s') \rangle_{\dR^3}$ we can express the square of the distance between $\s(r,s)$
and $\s(r',s')$ through $r$, $r'$, and $|\tau(s) - \tau(s')|$ as  
\begin{equation}\label{eq:geometric}
\begin{split}
	|\s(r,s) - \s(r',s')|^2 
	& =
	|r\tau(s) - r'\tau(s')|^2
	=
	r^2 + (r')^2 - 2rr'\langle\tau(s), \tau(s')\rangle_{\dR^3}\\
	&=
	r^2 + (r')^2 + rr'\big(|\tau(s) - \tau(s')|^2 -2\big)\\
	&=
	(r - r')^2 + rr'|\tau(s) - \tau(s')|^2.
\end{split}
\end{equation}	
In the next proposition we apply the BS-principle
and separation of variables to 
finite circular cones.
\begin{prop}\label{lem:BS2}
	Let $\cC\subset\dS^2$ be a circle and 
	let $\Gamma_R := \Sigma_R(\cC)$, $R\in(0,+\infty)$, 
	be as in~\eqref{def:S}.	Let the operator
	$\sfS_{\Gamma_R}(\kp)$, $\kp \ge 0$, be as in~\eqref{def:SL}.
	Then the eigenfunction corresponding 
	to the largest eigenvalue of $\sfS_{\Gamma_R}(\kp)$
	is rotationally invariant; \ie 
	it depends on the distance from the origin (tip of the cone) only.
\end{prop}


\begin{proof}
	Let $L \in (0,2\pi]$ stand for the length of $\cC$
	and $(\cdot,\cdot)_\cC$ denote the scalar product in $L^2(\cC)$.
	The family of functions
	\[
		\chi_m(s) 
		:= 
		\frac{1}{\sqrt{L}}
		\exp\bigg(\frac{2\pi m\ii s}{L}\bigg), 
		\qquad
		s\in (0,L),~m\in\dZ,
	\]
	constitutes an orthonormal basis in $L^2(\cC)$.
	The corresponding family of orthogonal projections
	\[
		\pi_m := \chi_m(\cdot,\chi_m)_\cC,\qquad m\in\dZ,
	\]
	in $L^2(\cC)$ induces through the isomorphism
	$L^2(\Gamma_R) \simeq L^2((0,R);r\dd r)\otimes L^2(\cC)$
	the family of orthogonal projections
	\[
		P_m := \fri \otimes \pi_m, \qquad m\in\dZ,
	\]
	in the Hilbert space $L^2(\Gamma_R)$ satisfying $P_mP_n = 0$ for $m\ne n$
	and $\sum_{m\in\dZ} P_m = I$; here $\fri$ and $I$ are the identity operators 
	in $L^2((0,R);r\dd r)$ and $L^2(\G_R)$, respectively.
	Thus, the decomposition
	\[
		L^2(\Gamma_R) = \bigoplus_{m\in\dZ} \ran P_m
	\]
	holds.
	Observe that for any $\psi \in L^2(\Gamma_R)$ and $m\in\dZ$ we obtain that
	$(P_m\psi)(r,s) = \varphi_m(r)\chi_m(s)$
	 with some $\varphi_m\in L^2((0,R);r\dd r)$.
	Note also that any $\psi \in \ran P_0$ is rotationally invariant
	and that any $\psi \in (\ran P_0)^\bot$ can not be positive on $\G_R$.

	Let $\tau\colon [0,L]\arr \dS^2$ be the unit-speed mapping parametrising
	$\cC$ and $\s\colon [0,R)\times [0,L]\arr \dR^3$ be the corresponding
	mapping parametrising $\Gamma_R$ as in~\eqref{eq:S_param}.
   	For the next argument it is convenient
	to extend the mapping $\tau$ periodically to the whole real-line.
	This extension defines also the corresponding
	extension of  $\s$ to $[0,R)\times \dR$. As there is 
	no danger of confusion, we employ the same notation for these extensions. 
	For any $s\in [0,L]$ and $t \in [0, L/2]$ we introduce the shorthand notation
	\[
		F_\kp(r,r',t) := 
		\frac{e^{-\kp |\s(r',s + t) - \s(r,s)|}}{4\pi
		 |\s(r',s + t) - \s(r,s )|}.
	\]
	A crucial point is
	that $F_\kp(r,r',t)$ does not depend on $s$ because of rotational
	invariance of $\Gamma_R$.	 
	Hence, for any $\psi \in L^2(\Gamma_R)$ and $m,n\in\dZ$, $m\ne n$, we get
	\begin{equation*}\label{eq:mn}
		(\sfS_{\Gamma_R}(\kp)P_n\psi, P_m\psi)_{\Gamma_R} 
		= 2
		\int_0^R\int_0^R \varphi_n(r) \ov{\varphi_m(r')} rr' \dd r \dd r' 
		\int_0^{L/2} F_\kp(r,r',t)\dd t 
		\int_0^L \chi_n(s)\ov{\chi_m(s+t)} \dd s.		
	\end{equation*}
	Thanks to $\int_0^L \chi_n(s)\ov{\chi_m(s + t)} \dd s =  0$, we end up with
	\begin{equation}\label{eq:mn}
		(\sfS_{\Gamma_R}(\kp)P_n\psi, P_m\psi)_{\Gamma_R} = 0.
	\end{equation}
	Thus, $P_n\sfS_{\Gamma_R}(\kp)P_m = 0$ for all $m,n\in\dZ$, $m\ne n$.
	Further, we define
	\[
		\sfS_{\Gamma_R}^{[m]}(\kp) := P_m\sfS_{\Gamma_R}(\kp)P_m,\qquad m\in\dZ.
	\]
	In view of~\eqref{eq:mn} we arrive at
	\[	
		\sfS_{\Gamma_R}(\kp) = \bigoplus_{m\in\dZ}\sfS_{\Gamma_R}^{[m]}(\kp).
	\]
	By Proposition~\ref{prop:BS}
	the eigenfunction corresponding to the largest
	eigenvalue of $\sfS_{\Gamma_R}(\kp)$ can be chosen to be positive almost
	everywhere on $\Gamma_R$. Thus, this eigenfunction necessarily
	belongs to $\ran P_0$ and the claim follows.
\end{proof}

\subsection{Energy of knots}
\label{ssec:knots}

Given a $C^2$-smooth loop $\cT \subset\dS^2$, 
$L := |\cT| \in (0,2\pi] $,
parametrized via the unit-speed mapping  $\tau \colon [0,L] \arr \dS^2$.
Let $|\tau(s) - \tau(t)|$ be the distance between
$\tau(s)$ and $\tau(t)$ in the ambient space $\dR^3$.
Let $f \in C([0,\infty);\dR)$ and consider the energy functional of the form
\begin{equation}\label{def:Phi}
	\Phi_f[\cT] 
	:= 
	\int_0^L\int_0^L f( |\tau(s) - \tau(t)|^2) \dd s \dd t.
\end{equation}
%

Finding the curve which minimizes this functional is a particular
problem in the theory of knots.  
The literature on knots and their energies is quite extensive;
see~\cite{ACGJ03, EFH07, EHL06, Lu66}, the monograph \cite{OHara03}
and the references therein. For our problem it is proven
that circles are unique minimizers under reasonable assumptions on $f$.
Below we formulate a specialized version of this result.
\begin{prop}\cite[Thm. 2.2]{EHL06}, \cite[Thm. 2]{ACGJ03}
	Let $f \in C([0,\infty);\dR)$ be convex and decreasing. 
	Let the functional $\Phi_f$ be as in~\eqref{def:Phi} with $L \in (0,2\pi]$.	
	Let $\cC\subset\dS^2$ be a circle and $\cT\subset \dS^2$ be a $C^2$-smooth 
	non-circular loop such that $|\cT| = |\cC|= L$. 
	Then the following isoperimetric inequality holds
	\[
		\Phi_f[\cC] < \Phi_f[\cT].
	\]
\label{thm:knot_energy}
\end{prop}

\section{Proofs of the main results}
\label{sec:proofs}

\subsection{Proof of Theorem~\ref{thm1}}
\label{ssec:proof1}
The claim of~(i) is an easy consequence of Proposition~\ref{prop:existence}
and, in particular, $\aa_{\rm cr} = 1/\mu_{\Gamma_R}(0)$,
where $\mu_{\Gamma_R}(0) > 0$
is the largest eigenvalue of $\sfS_{\Gamma_R}(0)$.

To show~(ii) consider the following auxiliary function $f\in C([0,+\infty);\dR)$ defined by
\begin{equation}\label{eq:f}
	f(x) := \frac{e^{-a\sqrt{bx+c}}}{\sqrt{bx+c}},
\end{equation}
where $a, b, c > 0$ are some parameters. By direct computation of derivatives we get
\[
\begin{split}
	f'(x)  &  
	 = -e^{-a\sqrt{bx+c}}
	  \bigg[\frac{ab}{2(bx+c)} +  \frac{b}{2(bx+c)^{3/2}}\bigg] < 0,   \\
	f''(x) &  =  
	e^{-a\sqrt{bx+c}}\bigg[ \frac{a^2b^2}{4(bx+c)^{3/2}} + 
	\frac{3ab^2}{4(bx+c)^2} + \frac{3b^2}{4(bx+c)^{5/2}}\bigg] > 0.
\end{split}
\]
Hence, the function $f$ satisfies the conditions of Proposition~\ref{thm:knot_energy}.
Further, let $\tau_\cC$ and $\tau_\cT$ be the unit-speed mappings
which parametrize $\cC$ and $\cT$, respectively.
Thus, for the functional $\Phi_f$ defined in~\eqref{def:Phi} 
with $L = |\cC| = |\cT|$ and $f$ as in~\eqref{eq:f}
we have by Proposition~\ref{thm:knot_energy} the following inequality
\begin{equation}\label{eq:Phi_f_ineq}
	\Phi_f[ \cC ] < \Phi_f[ \cT ].
\end{equation}
Next, let $\kp \in [0,+\infty)$ be fixed,
let $\mu_{\Gamma_R}(\kp) > 0$
be the largest eigenvalue of $\sfS_{\Gamma_R}(\kp)$
and $\psi_{\Gamma_R} \in L^2(\Gamma_R)$ be the
corresponding normalized eigenfunction
of $\sfS_{\Gamma_R}(\kp)$.
Let also $\mu_{\Lambda_R}(\kp) > 0$
be the largest eigenvalue of $\sfS_{\Lambda_R}(\kp)$.
By Propositions~\ref{prop:BS} and~\ref{lem:BS2} the function 
$\psi_{\Gamma_R}$ can be chosen to be positive and depending on the distance 
from the tip of the cone $\Gamma_R$ only. 
Let us introduce the following test function 
$\psi_{\Lambda_R} \colon \Lambda_R \arr \dR_+$ by
\[
	\psi_{\Lambda_R}(r,s) := \psi_{\Gamma_R}( r ), \qquad r\in (0,R).
\]	
Using the formula in~\eqref{eq:S_integral}, we verify that 
\[
	\|\psi_{\Lambda_R}\|^2_{\Lambda_R} 
	= 
	\int_0^L \int_0^R |\psi_{\Lambda_R}(r,s)|^2 r \dd r \dd s 
	= 
	\int_0^L \int_0^R |\psi_{\Gamma_R}(r)|^2 r \dd r \dd s 
	=
	\|\psi_{\Gamma_R}\|^2_{\Gamma_R} = 1.
\]	
Let $\s_{\Gamma_R}$ and $\s_{\Lambda_R}$ based on $\tau_\cC$ and $\tau_\cT$, respectively,
parametrize $\Gamma_R$ and $\Lambda_R$ as in~\eqref{eq:S_param}. 
Employing the identity~\eqref{eq:geometric}, we find
\begin{equation}\label{eq:geometric2}
\begin{split}
	|\s_{\Gamma_R}(r,s) - \s_{\Gamma_R}(r',s')|^2 & = 
	(r - r')^2 + rr'|\tau_\cC(s) - \tau_\cC(s')|^2,\\
	|\s_{\Lambda_R}(r,s) - \s_{\Lambda_R}(r',s')|^2 & = 
	(r - r')^2 + rr'|\tau_\cT(s) - \tau_\cT(s')|^2.\\
\end{split}
\end{equation}
Further, define for $\kp \ge 0$ the function
\begin{equation}\label{eq:Kkp}
	K_\kp(r,r') := \frac{rr'}{4\pi}\big(\Phi_f[ \cT ] - \Phi_f[ \cC ]\big),\qquad r,r' \in [0,R],
\end{equation}
where $\Phi_f$ is as in \eqref{def:Phi} and $f$ as in~\eqref{eq:f} with $a(r,r') := \kp$, $b(r,r') := rr'$, and  $c(r,r') := (r-r')^2$. 
Note that the function $K_\kp$ is not well-defined for $r = 0$, $r' = 0$ or $r = r'$, but
these conditions correspond to null subsets of $[0,R] \times [0,R]$ and can be neglected. 
By~\eqref{eq:Phi_f_ineq} we obtain that $K_\kp > 0$ almost everywhere on $[0,R] \times [0,R]$. 
Using~\eqref{def:SL},~\eqref{def:Phi},~\eqref{eq:f},~\eqref{eq:geometric2},~and~\eqref{eq:Kkp}
we get for all $\kp \ge 0$
\begin{equation*}
\begin{split}
	\mu_{\Lambda_R}(\kp) &\ge
	(\sfS_{\Lambda_R}(\kp)\psi_{\Lambda_R}, \psi_{\Lambda_R})_{\Lambda_R}
	=
	(\sfS_{\Gamma_R}(\kp)\psi_{\Gamma_R}, \psi_{\Gamma_R})_{\Gamma_R}
	+
	\int_0^R\int_0^R 
	K_\kp(r,r')\psi_{\Gamma_R}(r)\psi_{\Gamma_R}(r')\dd r \dd r'\\[0.3ex]
	&
	> (\sfS_{\Gamma_R}(\kp)\psi_{\Gamma_R},
	 \psi_{\Gamma_R})_{\Gamma_R} = \mu_{\Gamma_R}(\kp)
	 \|\psi_{\Gamma_R}\|^2_{\Gamma_R}
	 = \mu_{\Gamma_R}(\kp).
\end{split}	 
\end{equation*}
%
      
Now, let $\aa \ge \aa_{\rm cr} = (\mu_{\Gamma_R}(0))^{-1}$. 
Then we get by the above inequality for $\kp = 0$
that $\mu_{\Lambda_R}(0) > \mu_{\Gamma_R}(0) = 1/\aa_{\rm cr} \ge 1/\aa$
and using Proposition~\ref{prop:existence} we obtain $\#\sd(\sfH_{\aa,\Lambda_R}) \ge 1$.

Finally, let $\aa > \aa_{\rm cr}$. Then, by item~(i), $\#\sd(\sfH_{\aa,\G_R}) \ge 1$   
and set $\kp := (-E_1(\sfH_{\aa,\Gamma_R}))^{1/2} > 0$.
By Proposition~\ref{lem:BS1}\,(ii) we have $\mu_{\Gamma_R}(\kp) = 1/\aa$
and using again the above inequality
we end up with $\mu_{\Lambda_R}(\kp) > \mu_{\Gamma_R}(\kp) =  1 / \aa$.
Hence, by Proposition~\ref{lem:BS1}\,(i) we get 
\[
	E_1(\sfH_{\aa,\Lambda_R}) < -\kp^2 = E_1(\sfH_{\aa,\Gamma_R}).
\]

\subsection{Proof of Theorem~\ref{thm2}}
\label{ssec:proof2}
Before proving Theorem~\ref{thm2} itself, we provide two auxiliary
statements. The first statement is on the essential spectrum
in the case of infinite cones. It is proven in the circular case 
in~\cite[Thm. 2.1]{BEL14_JPA} and in a more general setting
in~\cite[Thm. 1.6]{BP15} using a different method.
\begin{prop}\label{prop:ess}
	Let $\Sigma_\infty := \Sigma_\infty(\cT)\subset\dR^3$ be an infinite $C^2$-smooth cone as in~\eqref{def:S} and let the self-adjoint operator
	$\sfH_{\aa,\Sigma_\infty}$ be as in Definition~\ref{dfn:Op}.
	Then $\inf\sess(\sfH_{\aa,\Sigma_\infty}) = -\aa^2/4$ holds.
\end{prop}
\begin{remark}
	Note that the lowest point of the essential spectrum 
	not necessarily equals to $-\aa^2/4$
	if the cross section of the underlying
	cone is not $C^2$-smooth \cite{BP15}. 
\end{remark}
Next, we formulate and prove a convergence lemma.
\begin{lem}\label{lem:conv}
	Let $\Sigma_R = \Sigma_R(\cT)\subset\dR^3$, $R\in(0,+\infty]$, be 
	$C^2$-smooth cones as in~\eqref{def:S}. 
	Let the self-adjoint operators
	$\sfH_{\aa,\Sigma_R}$ be as in Definition~\ref{dfn:Op}.
	Then $\sfH_{\aa,\Sigma_R}$ converge in the strong resolvent sense to 
	$\sfH_{\aa,\Sigma_\infty}$ as $R \arr +\infty$.
\end{lem}

\begin{proof}
	The sequence of quadratic forms $\fra_{\aa,\Sigma_R}$ as in~\eqref{def:frm}
	is monotonously decreasing in $R$ in the sense of ordering of forms. 
	For any $u \in H^1(\dR^3)$ we obtain by Lebesgue
	dominated convergence theorem that
	\[
		\fra_{\aa,\Sigma_R}[u] \arr 
		\fra_{\aa,\Sigma_\infty}[u],\qquad R\arr +\infty.
	\]
	The claim then follows from~\cite[Thm. S.16]{RS80-1}.
\end{proof}

\begin{proof}[Proof of Theorem~\ref{thm2}]
Let $\Gamma_R := \Sigma_R(\cC)$ and $\Lambda_R := \Sigma_R(\cT)$ 
with $R\in (0, +\infty)$ be $C^2$-smooth cones with the cross-sections $\cC$ and $\cT$, respectively. 
By Lemma~\ref{lem:conv} the operators $\sfH_{\aa,\Gamma_R}$ and 
$\sfH_{\aa,\Lambda_R}$ converge in the strong resolvent sense to 
the operators $\sfH_{\aa,\Gamma_\infty}$ and $\sfH_{\aa,\Lambda_\infty}$,
accordingly. Moreover, for any $R\in (0,\infty)$ 
the form orderings 
\[
	\fra_{\aa,\Gamma_\infty}\prec \fra_{\aa,\Gamma_R}
	\qquad\text{and}\qquad
	\fra_{\aa,\Lambda_\infty}\prec \fra_{\aa,\Lambda_R}
\]
can be directly verified. 
By Proposition~\ref{prop:ess} we also know that 
\begin{equation}\label{eq:essspec}
	\inf\sess(\sfH_{\aa,\Gamma_\infty}) = \inf\sess(\sfH_{\aa,\Lambda_\infty})
	=-\aa^2/4.
\end{equation}
Hence, using~\cite[Satz 9.26\,(b)]{W13} we obtain that
$\#\sd(\sfH_{\aa,\G_R}), \#\sd(\sfH_{\aa,\Lambda_R}) \ge 1$ for
all $R > 0$ large enough and that
\begin{equation}\label{eq:ineq_lim}
\begin{split}
	E_1(\sfH_{\aa,\Gamma_R})
	\arr \inf\s(\sfH_{\aa,\Gamma_\infty}),&\qquad R\arr +\infty,\\
	E_1(\sfH_{\aa,\Lambda_R}) \arr \inf\s(\sfH_{\aa,\Lambda_\infty}),&
	\qquad R\arr +\infty.
\end{split}
\end{equation}
By Theorem~\ref{thm1}\,(ii) we have the inequality
\begin{equation}\label{eq:ineq_R}
	E_1(\sfH_{\aa,\Lambda_R}) \le E_1(\sfH_{\aa,\Gamma_R})
\end{equation}
for all $R > 0$ so large that $\#\sd(\sfH_{\aa,\Gamma_R}) \ge 1$.		
Passing to the limit $ R \arr +\infty $ in the above inequality, 
and using~\eqref{eq:ineq_lim} we end up with
\begin{equation}\label{eq:ineq2}
	\inf\s(\sfH_{\aa,\Lambda_\infty}) 
	\le 
	\inf\s(\sfH_{\aa,\Gamma_\infty}).	
\end{equation}
%
Moreover, by~\cite[Thm. 3.2]{BEL14_JPA} holds $\sd(\sfH_{\aa,\Gamma_\infty}) \ne \varnothing$
and the point $\inf\s(\sfH_{\aa,\Gamma_\infty})$ is, in fact,
the lowest eigenvalue $E_1(\sfH_{\aa,\Gamma_\infty}) <  -\aa^2/4$ of 
$\sfH_{\aa,\Gamma_\infty}$. Thus, by~\eqref{eq:essspec} and~\eqref{eq:ineq2} 
the point $\inf\s(\sfH_{\aa,\Lambda_\infty})$ is,
respectively, the lowest eigenvalue $E_1(\sfH_{\aa,\Lambda_\infty}) <  -\aa^2/4$ of $\sfH_{\aa,\Lambda_\infty}$. 
Concluding,  the operator $\sfH_{\aa,\Lambda_\infty}$ 
has at least one bound state and the inequality~\eqref{eq:ineq2}
can be rewritten as stated in the theorem
\[
	E_1(\sfH_{\aa,\Lambda_\infty}) \le E_1(\sfH_{\aa,\Gamma_\infty}).	\qedhere	
\]
\end{proof}

\section*{Acknowledgments}
This research was supported by the Czech Science Foundation
(GA\v{C}R) within the project 14-06818S. 
We are grateful to the anonymous referee, whose suggestion inspired Remark~\ref{rem:strong}. 
\bibliographystyle{abbrv}


\end{document}

%% file: commands.tex
\newcommand\fri{\mathfrak{i}}

\def\arr{\rightarrow}

\def\aa{\alpha}
\def\lm{\lambda}

\def\s{\sigma}
\def\sd{\sigma_{\rm d}}
\def\sess{\sigma_{\rm ess}}

\def\ii{{\mathsf{i}}}

\def\kp{\kappa}

\def\sfH{\mathsf{H}}

\def\sfP{\mathsf{P}}
\def\dd{{\mathsf{d}}}

\def\sfS{\mathsf{S}}

\def\sfP{\mathsf{P}}

\newcounter{counter_a}
\newenvironment{myenum}{\begin{list}{{\rm(\roman{counter_a})}}%
{\usecounter{counter_a}
\setlength{\itemsep}{1.ex}\setlength{\topsep}{0.8ex}
\setlength{\leftmargin}{5ex}\setlength{\labelwidth}{5ex}}}{\end{list}}

\usepackage[latin1]{inputenc}
\usepackage[T1]{fontenc}

\newcommand{\eg}{{\it e.g.}\,}
\newcommand{\ie}{{\it i.e.}\,}
\newcommand{\cf}{{\it cf.}\,}

\numberwithin{figure}{section}
\numberwithin{equation}{section}
\theoremstyle{plain}
\newtheorem*{thm*}{Theorem}
\newtheorem{thm}{Theorem}[section]
\newtheorem{hyp}{Hypothesis}[section]
\newtheorem{lem}[thm]{Lemma}
\newtheorem{prop}[thm]{Proposition}

\newtheorem{dfn}[thm]{Definition}
\theoremstyle{remark}
\newtheorem{remark}[thm]{Remark}
\theoremstyle{plain}

\newcommand{\beu}{\begin{equation*}}
\newcommand{\eeu}{\end{equation*}}
\newcommand{\besu}{\begin{equation*}
\begin{aligned}}
\newcommand{\eesu}{\end{aligned}
\end{equation*}}
\newcommand{\bes}{\begin{equation}
\begin{aligned}}
\newcommand{\ees}{\end{aligned}
\end{equation}}

\newcommand\cM{\mathcal M}

\newcommand\fra{\mathfrak a}

\newcommand\ov{\overline}

\newcommand\void[1]{}

\def\ov{\overline}

\def\ran{{\rm ran\,}}

      \def\dC{{\mathbb C}}

      \def\dR{{\mathbb R}}
\def\dS{{\mathbb S}}      
      
   \def\dZ{{\mathbb Z}}

      \def\cC{{\mathcal C}}
      
      \def\cI{{\mathcal I}}
      
\def\cM{{\mathcal M}}      
      
   \def\cT{{\mathcal T}}

\newcommand{\dom}{\mathrm{dom}\,}